\documentclass[pdflatex,sn-mathphys-num]{sn-jnl}

\usepackage{graphicx}%
\usepackage{multirow}%
\usepackage{amsmath,amssymb,amsfonts}%
\usepackage{amsthm}%
\usepackage{mathrsfs}%
\usepackage[title]{appendix}%
\usepackage{xcolor}%
\usepackage{textcomp}%
\usepackage{manyfoot}%
\usepackage{booktabs}%
\usepackage{algorithm}%
\usepackage{algorithmicx}%
\usepackage{algpseudocode}%
\usepackage{listings}%
\usepackage{multirow}

\theoremstyle{thmstyleone}%
\newtheorem{theorem}{Theorem}
\newtheorem{proposition}[theorem]
{Proposition}%
\newtheorem{lemma}[theorem]
{Lemma}
\newtheorem{corollary}[theorem]
{Corollary}

\theoremstyle{thmstyletwo}%

\theoremstyle{thmstylethree}%

\begin{document}

\title[On the Wiener-like root-indices of graphs]{On the Wiener-like root-indices of graphs}


\author[1,7]{\fnm{Simon} \sur{Brezovnik}}\email{simon.brezovnik@fs.uni-lj.si}
\equalcont{These authors contributed equally to this work.}

\author[2,3,4,5]{\fnm{Matthias} \sur{Dehmer}}\email{matthias.dehmer@ffhs.ch}
\equalcont{These authors contributed equally to this work.}

\author*[6,7]{\fnm{Niko} \sur{Tratnik}}\email{niko.tratnik@um.si}
\equalcont{These authors contributed equally to this work.}

\author[6,8]{\fnm{Petra} \sur{Žigert Pleteršek}}\email{petra.zigert@um.si}
\equalcont{These authors contributed equally to this work.}

\affil[1]{\orgname{University of Ljubljana}, \orgdiv{Faculty of Mechanical Engineering}, \orgaddress{\city{Ljubljana}, \country{Slovenia}}}

\affil[2]{ \orgname{Swiss Distance University of Applied Sciences}, \orgaddress{\city{Brig}, \country{Switzerland}}}

\affil[3]{\orgdiv{Department for Biomedical Computer Science and Mechatronics}, \orgname{UMIT—Private University for Health Sciences}, \orgaddress{\city{Hall in Tyrol}, \country{Austria}}}

\affil[4]{\orgname{Nankai University}, \orgdiv{College of Artificial Intelligence}, \orgaddress{ \city{Tianjin}, \country{China}}}

\affil[5]{\orgname{Xian Technological University}, \orgdiv{School of Science}, \orgaddress{\street{Street}, \city{Xian}, \state{Shaanxi}, \country{China}}}

\affil[6]{\orgname{University of Maribor}, \orgdiv{Faculty of Natural Sciences and Mathematics}, \orgaddress{\city{Maribor}, \country{Slovenia}}}

\affil[7]{\orgname{Institute of Mathematics, Physics, and Mechanics}, \orgaddress{\city{Ljubljana}, \country{Slovenia}}}

\affil[8]{\orgname{University of Maribor}, \orgdiv{Faculty of Chemistry and Chemical Engineering}, \orgaddress{\city{Ljubljana}, \country{Slovenia}}}


\abstract{In this paper, we examine roots of graph polynomials where those roots can be considered as structural graph measures.
More precisely, we prove analytical results for the roots of certain modified graph polynomials and also discuss numerical results. As polynomials, we use, e.g.,
the Hosoya, the Schultz, and the Gutman polynomial which belong to an interesting family of degree-distance-based graph polynomials; they constitute so-called counting polynomials with non-negative integers  
as coefficients and the roots of their modified versions have been used to characterize the topology of graphs. Our results can be applied for the quantitative  characterization of graphs. Besides analytical results, we also investigate other properties of those measures such as their degeneracy which is an undesired aspect of graph measures.
It turns out that the measures representing roots of graph polynomials possess high discrimination power on exhaustively generated trees, which outperforms standard versions of these indices. Furthermore, a new measure is introduced that allows us to compare different topological indices in terms of structure sensitivity and abruptness.}




\maketitle

\section{Introduction}

{{In network theory and mathematical chemistry, topological indices \cite{devillers_book,todeschini_2002} have been
used to quantitatively  describe the structure of graphs.
The origin of this research field goes 75 years back when the famous Wiener index was introduced, see \cite{wiener}.
The Wiener index has been used in graph theory \cite{klavzar_1996} as well as in drug design and related
fields such as computational chemistry \cite{basak}.
Later, several other graph-based indices were invented, for example the edge-Wiener index \cite{dankelmann-2009,iran-2009,khalifeh-2009},
the Schultz index (also known as the degree-distance) \cite{dobrynin,schultz}, and the Gutman index \cite{gut_in}.
Another categorization of graph measures is given in \cite{EmmertStreib2011NetworksFS} that distinguishes between several
categories. In this paper, we consider graph measures which are based on graph polynomials.
Dehmer et al. \cite{dehmer_plos_2010} started this development by defining a special graph polynomial and
use the moduli of the zeros of the underlying graph polynomial as graph measures; it turned out that the measures have low
degeneracy. After this, Dehmer et al.  defined the so-called orbit polynomial where it's coefficients
are based on vertex orbit cardinalities and proved several properties and bounds thereof, see  \cite{dehmer, dehmer_emmert}.
In fact, the orbit polynomial has been used as an efficient symmetry measure for graphs, see \cite{dehmer,ghorbani_dehmer,Ma}.
Other contributions of Dehmer et al. who started the research on this topic can be found in
\cite{dehmer_moosbrugger_shi_2015}. Recent contributions to this field are due to Ghorbani et al. \cite{ghorbani1,ghorbani} and  Brezovnik et al. \cite{BDTZ23}. In the latter paper, measures based on roots of the Szeged and Mostar polynomials have been examined. 
}

{{Here we continue this line of research and therefore investigate some new  graph measures based on certain roots of graph polynomials.
These roots are positive, real-valued and lie in the interval $(0,1]$. It's evident that they have a certain structural interpretation
when measuring the complexity of graphs quantitatively. A well-known example is the unique, positive root $\delta \in (0,1]$ of the
modified orbit polynomial which has been proven useful as a symmetry measure.
In order to define and examine new root-based measures (so-called root-indices), we here focus on the Hosoya polynomial \cite{hosoya},
the edge-Hosoya polynomial \cite{behm}, the Schultz polynomial \cite{gut_pol}, and the Gutman polynomial \cite{gut_pol}.

In this paper, we firstly consider closed formulas for the mentioned polynomials of some well-known deterministic
graph classes. Then, bounds for the roots of the modified versions of considered polynomials are proved. Next, we compute various measures for evaluating the quality of different topological indices. In particular,
we consider the discrimination power of root-indices and compare their performance with that of existing
similar indices. Next, we consider correlations between different pairs of root-indices and also correlations
between them and their standard versions. This analysis reveals 
interesting links to other existing graph measures.
 Finally, we investigate structure sensitivity and abruptness which
measure how a gradual change of a graph results on the topological index. It is preferred that the first one is
as large as possible and the second as small as possible. Therefore, we introduce a new measure $SA$ as their
quotient, which enables us to efficiently compare different topological indices. 
}

\section{Preliminaries}
\label{prel}

Let  $G$ be a connected graph with at least one edge. Moreover, denote by $d_G(a,b)$  the standard shortest-path \textit{distance} between vertices $a, b \in V(G)$. In addition, let $\deg(a)$ be the \textit{degree} of  vertex $a$. The distance $d(e_1,e_2)$ between edges $e_1$ and $e_2$ of graph $G$ is  the distance between the corresponding vertices $e_1$ and $e_2$ in the line graph $L(G)$ of $G$.

\noindent
The \textit{Hosoya polynomial} of a graph $G$, denoted as $H(G,x)$, and the \textit{edge-Hosoya polynomial} of $G$, denoted as $H_e(G,x)$, are defined as
\begin{eqnarray*}
H(G,x) &=& \sum_{\substack{\{ a,b \} \subseteq V(G) \\ a \neq b}}  x^{d(a,b)},\\
H_e(G,x) &=& \sum_{\substack{\{ e,f \} \subseteq E(G) \\ e \neq f}} x^{d(e,f)}.\\
\end{eqnarray*}
Note that usually these polynomials include a nonzero constant term (the number of vertices in the Hosoya polynomial and the number of edges in the edge-Hosoya polynomial). However, the original definition of the Hosoya polynomial does not consider pairs of vertices $\{ a, b \}$ in which $a=b$. We use the latest definition since it is more suitable for our purposes.

It is clear that
$$H_e(G,x) = H(L(G),x),$$
where $L(G)$ is the line graph of $G$.

Moreover, the \textit{Schultz polynomial} of $G$, denoted as $Sc(G,x)$, and the \textit{Gutman polynomial} of $G$, denoted as $Gut(G,x)$, are defined as
\begin{eqnarray*}
Sc(G,x) &=& \sum_{\substack{\{ a,b \} \subseteq V(G) \\ a \neq b}} (\deg(a) + \deg(b)) x^{d(a,b)},\\
Gut(G,x) &=& \sum_{\substack{\{ a,b \} \subseteq V(G) \\ a \neq b}} \deg(a) \deg(b) x^{d(a,b)}.\\
\end{eqnarray*}

Obviously, the \textit{Wiener index}, the \textit{edge-Wiener index}, the \textit{Schultz index}, and the \textit{Gutman index} can be computed from the mentioned polynomials by evaluating their first derivatives at $x = 1$:
\begin{eqnarray*}
W(G) = H'(G,1), & & W_e(G) = H_e'(G,1),\\
 Sc(G) = Sc'(G,1), & & Gut(G) = Gut'(G,1).
\end{eqnarray*}

{ Finally, we show an alternative way for writing the polynomials defined above. For a graph $G$ and $k \geq 1$, we  define the following sets of unordered pairs of vertices and edges of $G$:
\begin{eqnarray*}
V_k(G) &=& \{ \{ a,b \} \subseteq V(G) \ | \ d(a,b) =k \},\\
E_k(G) &=& \{ \{ e,f \} \subseteq E(G) \ | \ d(e,f) =k \}.
\end{eqnarray*}

\noindent
In addition, for every $k \geq 1$ we set
\begin{eqnarray*}
d_k(G) & = & \left| V_k(G) \right|,\\
d^e_k(G) &=& \left| E_k(G) \right|, \\
s_k(G) & = & \sum_{ \{ a,b \} \in V_k(G)} \left( \deg(a) + \deg(b) \right),\\
g_k(G) & = & \sum_{ \{ a,b \} \in V_k(G)}  \deg(a) \deg(b).
\end{eqnarray*}
Note that if $V_k(G)$ is an empty set for some $k \geq 1$, then we define $s_k(G)=g_k(G)=0$.

It is straightforward to see that the following holds:
\begin{eqnarray*}
H(G,x) = \sum_{k \geq 1} d_k(G) x^k, & & H_e(G,x) = \sum_{k \geq 1} d^e_k(G) x^k,\\
 Sc(G,x) = \sum_{k \geq 1} s_k(G) x^k, & & Gut(G,x) = \sum_{k \geq 1} g_k(G) x^{k}.
\end{eqnarray*}

The following numbers will be  used in Section \ref{gl_se}:
\begin{eqnarray*}
MH(G) = \max \{ d_k(G) \ | \ k \geq 1\}, & & MH_e(G) = \max \{ d^e_k(G) \ | \ k \geq 1\},\\
MS(G) = \max \{ s_k(G) \ | \ k \geq 1\}, & & MG(G) = \max \{ g_k(G) \ | \ k \geq 1\}.
\end{eqnarray*}}

\section{Explicit formulas for some degree-distance-based polynomials}
\label{sec3}

Here we state the explicit formulas for the Hosoya polynomial, the edge-Hosoya polynomial, the Schultz polynomial, and the Gutman polynomial for some  families of graphs. It should be pointed out that some of the results  are quite straightforward and were already stated in earlier papers, for example see \cite{behm}.

\begin{proposition} \label{prop1}
For the complete graph $K_n$ with $n$ vertices, $n \geq 1$, it holds
\begin{eqnarray*}
H(K_n,x) & = & \frac{n(n-1)}{2} x, \\
Sc(K_n,x)&  = & n(n-1)^2x,\\
Gut(K_n,x) &=&\frac{n(n-1)^3}{2} x.
\end{eqnarray*}
\end{proposition}

\begin{proposition} \label{prop2}
For the complete graph $K_n$ with $n$ vertices, $n \geq 1$, it holds
$$H_e(K_n,x)= \frac{n^3 - 3n^2 + 2n}{2}x + \frac{n^4 - 6n^3 +11n^2 -6n}{8}x^2.$$
\end{proposition}

\begin{proof}
It can be easily checked that the formula holds for $n=2$ and $n=3$. Therefore, let $n \geq 4$. If two distinct edges $e=ab$ and $f=uv$ of $K_n$ are not adjacent, then $d(e,f)=2$ since for example $bu$ is an edge in $K_n$. Consequently, the diameter of $L(K_n)$ equals two.

Moreover, every edge of $K_n$ is adjacent to exactly $2(n-2)$ other edges, and therefore, the number of pairs of edges at distance one is
$$ \frac{{n \choose 2} 2(n-2)}{2} = \frac{n^3 - 3n^2 + 2n}{2}.$$
Hence, the number of pairs of edges at distance two is
$${ {n \choose 2} \choose 2} -  \frac{n^3 - 3n^2 + 2n}{2} = \frac{n^4 - 6n^3 +11n^2 -6n}{8},$$
which completes the proof.
 \end{proof}

\begin{proposition} \label{prop3}
For the cycle $C_n$ with $n$ vertices, $n \geq 3$, it holds
\begin{eqnarray*}
H(C_n,x) =H_e(C_n,x) & = & \begin{cases} nx \left(1 + x + \cdots + x^{\frac{n}{2}-2} + \frac{1}{2} x^{\frac{n}{2}-1} \right); & n \textrm{ even} \\
nx\left( 1 + x + \cdots +  x^{\frac{n-3}{2}}\right) ; & n \textrm{ odd}
\end{cases}, \\
Sc(C_n,x) = Gut(C_n,x) &  = & \begin{cases} 4nx \left(1 + x + \cdots + x^{\frac{n}{2}-2} + \frac{1}{2} x^{\frac{n}{2}-1} \right); & n \textrm{ even} \\
4nx\left( 1 + x + \cdots +  x^{\frac{n-3}{2}}\right) ; & n \textrm{ odd}
\end{cases}.
\end{eqnarray*}

\end{proposition}

\begin{proposition} \label{prop4}
For the star $S_n$ with $n +1$ vertices, $n \geq 1$, it holds
\begin{eqnarray*}
H(S_n,x) & = & nx + \frac{n(n-1)}{2}x^2, \\
Sc(S_n,x)&  = & n(n+1)x + n(n-1)x^2,\\
Gut(S_n,x) &=& n^2x + \frac{n(n-1)}{2}x^2,\\
H_e(S_n,x) &=&  \frac{n(n-1)}{2}x.
\end{eqnarray*}
\end{proposition}

\begin{proposition} \label{prop5}
For the wheel $W_n$ with $n +1$ vertices, $n \geq 3$, it holds
\begin{eqnarray*}
H(W_n,x) & = & 2nx + \frac{n(n-3)}{2}x^2, \\
Sc(W_n,x)&  = & n(n+9)x + 3n(n-3)x^2,\\
Gut(W_n,x) &=& 3n(n+3) x + \frac{9n(n-3)}{2}x^2,\\
H_e(W_n,x) &=&  \frac{n(n+5)}{2}x + n(n-1)x^2 + \frac{n(n-5)}{2}x^3.
\end{eqnarray*}
\end{proposition}

\begin{proposition} \label{prop6}
For the path $P_n$ with $n$ vertices, $n \geq 3$, it holds
\begin{eqnarray*}
H(P_n,x) & = &  \sum_{i=1}^{n-1} (n-i)x^i, \\
Sc(P_n,x) &=& \sum_{i=1}^{n-2} (4n-2-4i)x^i + 2x^{n-1}, \\
Gut(P_n,x) & = & \sum_{i=1}^{n-2} 4(n-1-i)x^i + x^{n-1}, \\
H_e(P_n,x) & = & \sum_{i=1}^{n-2} (n-1-i)x^i.
\end{eqnarray*}
\end{proposition}

\section{Roots of modified degree-distance-based polynomials}
\label{gl_se}

In this section, we investigate some newly defined polynomials. In particular, we introduce the following modified polynomials of a connected graph $G$:
\begin{eqnarray*}
H^*(G,x) & = & 1- H(G,x), \\
Sc^*(G,x) & = & 1- Sc(G,x), \\
Gut^*(G,x) & = & 1- Gut(G,x), \\
H_e^*(G,x)&  = & 1- H_e(G,x).
\end{eqnarray*}

The motivation for studying these polynomials comes from a result used in  \cite{dehmer}, see also \cite{BDTZ23}.
\begin{lemma} \cite{BDTZ23,dehmer}  \label{main_le}
Let $Q(x)= q_1x + q_2 x^2 + \cdots + q_n x^n$ be a polynomial, $n \in \mathbb{N}$, $q_i \in [0, \infty)$ for all $i \in \{ 1, \ldots, n \}$, and $q_1 + q_2 + \cdots + q_n \geq 1$. Then the polynomial $Q^*(x) = 1- Q(x)$ has a unique positive root $\delta$ such that $\delta \in (0,1]$. { Moreover, $\delta = 1$ if and only if $q_1 + q_2 + \cdots + q_n = 1$.}
\end{lemma}

Obviously, if a connected graph has at least two vertices, then  the polynomials $H$, $Sc$, and $Gut$ can be written as $q_1x + q_2 x^2 + \cdots + q_n x^n$, where $n \in \mathbb{N}$, $q_i \in [0, \infty)$ for all $i \in \{ 1, \ldots, n \}$, and $q_1 + q_2 + \cdots + q_n \geq 1$. However, for the edge-Hosoya polynomial $H_e$ we consider only connected graphs with at least three vertices, since otherwise $H_e(G,x)=0$.

According to Lemma \ref{main_le}, each polynomial $Q^* \in \{H^*,H_e^*,Sc^*,Gut^* \}$ has a unique positive root, denoted as $\delta(Q^*)$, which always lies within the interval $(0,1]$. This positive root serves as a topological index, referred to as a \textit{root-index} \cite{BDTZ23}. In the rest of the section, we focus on some theoretical results regarding the obtained new root-indices of a graph $G$: the \textit{Wiener root-index} $\delta(H^*(G,x))$, the \textit{edge-Wiener root-index} $\delta(H_e^*(G,x))$, the \textit{Schultz root-index} $\delta(Sc^*(G,x))$, and the \textit{Gutman root-index} $\delta(Gut^*(G,x))$.

{ The next proposition follows by Lemma \ref{main_le}.
\smallskip

\begin{proposition}
Suppose $G$ is a connected graph with at least one edge. If
$Q^* = H^*$ or $Q^* = Gut^*$, then $\delta(Q^*(G,x))=1$ if and only if $G$ is isomorphic to $P_2$. Moreover, $\delta(Sc^*(G,x)) < 1$. Furthermore, if $G$ is a connected graph with at least two edges, then $\delta(H_e^*(G,x))=1$ if and only if $G$ is isomorphic to $P_3$.
\end{proposition}

\begin{proof} Obviously, $\displaystyle \sum_{k \geq 1} d_k(G)= 1$ if and only if $G$ has only one pair of vertices, which is further equivalent to $G$ being isomorphic to $P_2$. Analogous reasoning gives us the same conclusion for $\displaystyle \sum_{k\geq 1} g_k(G)$. On the other hand, since $s_1 \geq 2$, we have $\displaystyle \sum_{k\geq 1} s_k(G) > 1$, which implies $\delta(Sc^*(G,x)) < 1$. Finally, $\displaystyle \sum_{k \geq 1} d_k^e(G)= 1$ implies $d_1^e=1$, which is true only if $G$ has exactly two edges, so it should be isomorphic to $P_3$. Consequently, the stated proposition follows by Lemma \ref{main_le}.
\end{proof}}

{ In what follows, we establish lower bounds for all  investigated root-indices.
\smallskip

\begin{theorem} \label{th_do}
Let $G$ be a connected graph with at least three vertices. Then
\begin{eqnarray*}
\delta(H^*(G,x)) > \frac{1}{MH(G) + 1}, & & \delta(Sc^*(G,x)) > \frac{1}{MS(G) + 1},\\
\delta(Gut^*(G,x)) > \frac{1}{MG(G) + 1}, & & \delta(H_e^*(G,x)) > \frac{1}{MH_e(G) + 1}.
\end{eqnarray*}
\end{theorem}
\begin{proof}
The proof is analogous for all polynomials, so we consider only $\delta(H^*(G,x))$. Let $\delta = \delta(H^*(G,x))$. Obviously, the result holds trivially if $\delta =1$, so in the following suppose that $\delta < 1$. We consider
\begin{equation*}
H(G,\delta)= \sum_{k \geq 1} d_k(G) \delta^k,
\end{equation*}
where  $d_k(G)$ is the number of pairs of vertices at distance $k$.
We start with
\begin{equation*}
H(G,\delta)=\sum_{k \geq 1} d_k(G) \delta^k  < \sum_{k \geq 1} MH(G) \delta^k =  MH(G) \sum_{k = 1}^{\infty}  \delta^k = MH(G)\frac{\delta}{ 1- \delta}.
\end{equation*}
Consequently, it holds
\begin{equation} \label{eq_pr1}
\sum_{k \geq 1} d_k(G) \delta^k  <  MH(G)\frac{\delta}{ 1- \delta}.
\end{equation}
However, it is clear that
$$H^*(G,\delta) = 0,$$
so from $H^*(G,\delta) = 1 - H(G,\delta)$ we also get
\begin{equation} \label{eq_pr2}
H(G,\delta)=\sum_{k \geq 1} d_k(G) \delta^k =1.
\end{equation}
Finally,  Equations \eqref{eq_pr1} and \eqref{eq_pr2} imply
\begin{equation*}
1 < MH(G)\frac{\delta}{ 1- \delta},
\end{equation*}
which gives
   $$\delta(H^*(G,x)) > \frac{1}{MH(G) + 1}.$$
\end{proof}

For demonstration, we will compute the polynomials, root-indices, and lower bounds of a chemical graph  $G$ from Figure \ref{fig1}.

\begin{figure}[!htb]
	\centering
		\includegraphics[scale=0.7, trim=0cm 0cm 1cm 0cm]{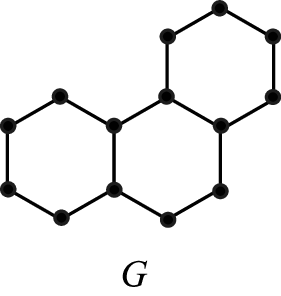}
\caption{Chemical graph $G$ of phenanthrene.}
	\label{fig1}
\end{figure}

\noindent
The corresponding  polynomials of $G$ are
\begin{eqnarray*}
H(G,x) &=& x^7 + 4x^6 + 10x^5 + 16x^4 + 22x^3 + 22x^2 + 16x, \\
H_e(G,x) &=& 4x^6 + 10x^5 + 20x^4 + 33x^3 + 31x^2 + 22x, \\
Sc(G,x) &=& 4x^7 + 16x^6 + 42x^5 + 70x^4 + 102x^3 + 106x^2 + 76x, \\
Gut(G,x) &=& 4x^7 + 16x^6 + 44x^5 + 76x^4 + 117x^3 + 126x^2 + 91x.
\end{eqnarray*}
Consequently, we obtain the root-indices of $G$:
\begin{eqnarray*}
\delta(H^*(G,x)) \doteq 0.05765, & & \delta(H_e^*(G,x)) \doteq 0.04276,\\
\delta(Sc^*(G,x)) \doteq 0.01292, & & \delta(Gut^*(G,x)) \doteq 0.01082.
\end{eqnarray*}

\noindent
However, from the polynomials follows that
$$MH(G) = 22, \quad MH_e(G)= 33, \quad MS(G)= 106, \quad MG(G)= 126.$$
Hence, by Theorem \ref{th_do} we get
\begin{eqnarray*}
\delta(H^*(G,x)) > \frac{1}{23}, & & \delta(H_e^*(G,x)) > \frac{1}{34},\\
\delta(Sc^*(G,x)) > \frac{1}{107}, & & \delta(Gut^*(G,x)) > \frac{1}{127}.
\end{eqnarray*}
We also notice that the lower bounds from Theorem \ref{th_do} give  good approximations for the  root-indices of graph $G$.

In the rest of the section, we focus on the root-indices of basic graph families.} To show their behaviour for graphs with large order, we firstly prove the following lemma.
\smallskip

\begin{lemma} \label{pomozna}
For any $n \geq 1$ let $Q_n = q_{n,1} x + q_{n,2} x^2 + \cdots + q_{n,m_n}x^{m_n}$ be a polynomial that satisfies the conditions of Lemma  \ref{main_le}. Moreover, suppose that there exists a constant $k \geq 1$ and a polynomial $p$ with degree at least one such that for every $n \geq 1$ it holds $q_{n,k}=p(n)$. If $c_n = \delta(Q_n^*)$ for all $n \geq 1$, then the sequence $(c_n)$ converges to 0.
\end{lemma}

\begin{proof}
Obviously, for any $n \geq 1$ it holds
$$Q^*(c_n) = 1- Q_n(c_n) =0$$
and therefore
$$Q_n(c_n) = q_{n,1} c_n + q_{n,2} c_n^2 + \cdots + q_{n,m_n}c_n^{m_n} =1.$$
Since the polynomial $Q_n$ satisfies the conditions of Lemma \ref{main_le} it follows that
$$0 \leq q_{n,k}c_n^k=p(n)c_n^k \leq 1$$
for all $n \geq 1$. Consequently, there exists $n_0 \geq 1$ such that for any $n \geq n_0$ it holds
$$0 \leq c_n^k \leq \frac{1}{p(n)}  \ \implies \  0 \leq c_n \leq \frac{1}{\sqrt[k]{p(n)}}.$$
Since
$$\lim_{n \to \infty} \frac{1}{\sqrt[k]{p(n)}} =0,$$
we have proved that the sequence $(c_n)$ converges to 0.
\end{proof}

\begin{corollary}
Let $Q \in \{ H, Sc, Gut, H_e \}$ be a polynomial and for any $n \geq 3$ let $G_n \in \{K_n, C_n, S_n, W_n, P_n \}$. Then we have
$$\lim_{n \to \infty} \delta(Q^*(G_n,x)) =0.$$
\end{corollary}

\begin{proof}
The result follows directly from Lemma \ref{pomozna} and Propositions \ref{prop1} to \ref{prop6}.
\end{proof}

By using Proposition \ref{prop1} we can also obtain explicit formulas for  three root-indices of complete graphs.
\smallskip

\begin{proposition}
For the complete graph $K_n$ on $n$ vertices, $n \geq 2$ vertices, it follows
\begin{eqnarray*}
\delta(H^*(K_n,x)) & = & \frac{2}{n(n-1)} , \\
\delta(Sc^*(K_n,x))&  = & \frac{1}{n(n-1)^2},\\
\delta(Gut^*(K_n,x)) &=&  \frac{2}{n(n-1)^3}.
\end{eqnarray*}
\end{proposition}

\section{{Numerical results}}
\label{sec5}

Various numerical results related to the Wiener, Gutman, Schultz, and edge-Wiener root-indices are presented and evaluated in this section. With $N_j$ and $T_j$ we denote the families of all connected graphs and all trees on $j$ vertices, respectively.

Firstly, we consider the discrimination power of root-indices. In particular, we investigate the discrimination (also called sensitivity) which was introduced in \cite{konst}. Let $\mathcal{F}$ be a finite family of graphs that are pairwise non-isomorphic. Moreover, let $I$ be a topological index and  $\mathcal{C}$ the set of all  graphs $G$ from $\mathcal{F}$ for which there exists a graph $H \in \mathcal{F}$ such that $G$ and $H$  are not isomorphic but $I(G)=I(H)$. With other words, $\mathcal{C}$ is the set of graphs that cannot be distinguished by  topological index $I$. In addition, let  $ND = |\mathcal{C}|$. The \textit{discrimination} $Dis(I)$ of a topological index $I$ is defined as
$$Dis(I) = \frac{|\mathcal{F}| - ND}{|\mathcal{F}|}.$$

For the families \(N_j\), where \(j \in \{ 5, 6, 7, 8 \}\), and \(T_j\), where \(j \in \{ 8, \ldots, 16 \}\), we have computed $ND$ and $Dis$ for all four investigated root-indices, see Table \ref{tab_sen}.

\begin{table}[!ht]
\caption{\label{tab_sen} Discrimination for considered root-indices  in different families of graphs.}
{\small{
	\begin{tabular} {ccllllllll}
	\toprule
	&& \multicolumn{2}{c}{$\delta(H^{\ast})$} & \multicolumn{2}{c}{$\delta(Gut^{\ast})$}&\multicolumn{2}{c}{$\delta(Sc^{\ast})$}&\multicolumn{2}{c}{$\delta(H_e^{\ast})$}\\
	\midrule
	family& no.\ of graphs  & $ND$& $Dis$& $ND$& $Dis$ & $ND$& $Dis$ & $ND$& $Dis$   \\
	\midrule
  \multicolumn{10}{c}{families of  connected graphs}\\
  \midrule	
 $N_5$&  21  &   14    &  0.3333  &   0   &  1.0000   &  4    &   0.8095   &  4   & 0.8095  \\
 $N_6$& 112 &   94    & 0.1607   &   14  &  0.8750   &  51    &   0.5446   &  62 &  0.4464 \\

 $N_7$& 853 &  811   & 0.0492  &    279   &  0.6729   &  588  &  0.3107    &  694  & 0.1864  \\
 $N_8$& 11117&11014 & 0.0093  &  7010 &   0.3694  &  9823    &  0.1164 &   10523   &  0.0534 \\

  \midrule
  \multicolumn{10}{c}{families of   trees  }\\
  \midrule
   $T_{8}$& 23   & 0  & 1.0000  & 0  & 1.0000    &  0 &  1.0000     & 0   & 1.0000   \\
  $T_{9}$& 47    & 4  &  0.9149 & 4  & 0.9149  & 4  &    0.9149  & 4 & 0.9149    \\
  $T_{10}$& 106  & 4  &  0.9623 & 4  & 0.9623  & 4 &  0.9623  & 4  &   0.9623 \\
  $T_{11}$& 235  & 39 & 0.8340  & 39 & 0.8340  & 39 & 0.8340  & 39 &   0.8340\\
  $T_{12}$& 551  & 58 & 0.8947  & 58  & 0.8947  & 58 &  0.8947  & 58  & 0.8947  \\
  $T_{13}$& 1301 & 214 & 0.8355 & 214 & 0.8355 & 214 & 0.8355 & 214 & 0.8355 \\
  $T_{14}$& 3159&  498 & 0.8424 &  498& 0.8424 &  498 & 0.8424&  498& 0.8424  \\
  $T_{15}$&7741 &1609 & 0.7921 & 1609 & 0.7921  & 1611 & 0.7919 & 1609 & 0.7921  \\
  $T_{16}$& 19320 & 3873& 0.7995  & 3877 & 0.7993  & 3889 & 0.7987 & 3873 &  0.7995 \\
  \bottomrule
\end{tabular}
}}
\end{table}

We can conclude that the Gutman root-index has the best discrimination power among all connected graphs. For example, this index can completely discriminate connected graphs on 5 vertices and it can discriminate more than 87\% of connected graphs on 6 vertices, while on the other hand the Wiener root-index can discriminate only approximately 16\% of them. As it turns out,  the worst discrimination ability in all the mentioned families of connected graphs is obtained for the Wiener root-index. However, in the families of trees all four root-indices have almost the same discrimination power, which is very high. For example, they can distinguish all trees on 8 vertices and almost 80\% of trees on 16 vertices.

To validate the introduction of new structural descriptors, we compare their performance with that of existing similar descriptors. We calculate the discrimination for all four standard versions of the corresponding indices: the Wiener index, the Gutman index, the Schultz index, and the edge-Wiener index, see Table \ref{tab_sen_stari}.

\small{\begin{table}[!ht]
	\caption{\label{tab_sen_stari} Discrimination for considered standard indices  in different families of graphs.}
\begin{tabular} {ccllllllll}
	\toprule
	&& \multicolumn{2}{c}{$W$} & \multicolumn{2}{c}{$Gut$}&\multicolumn{2}{c}{$Sc$}&\multicolumn{2}{c}{$W_e$}\\
	\midrule
	family& no.\ of graphs  & $ND$& $Dis$& $ND$& $Dis$ & $ND$& $Dis$ & $ND$& $Dis$   \\
	\midrule
  \multicolumn{10}{c}{families of  connected graphs}\\
  \midrule	
 $N_5$&  21  &   16   &  0.2381 &   2   &  0.9048  &  8  &   0.6190   &  7   & 0.6667  \\
  $N_6$& 112 &   108   & 0.0357   &   41  &  0.6339   &  96    &   0.1429 &  91 &  0.1875 \\
 $N_7$& 853 &  847   & 0.0070  &    715  &  0.1618  &  826  &  0.0317   &  829  & 0.0281  \\
 $N_8$& 11117& 11110 & 0.0006  &  10998 &   0.0107  &  11099    &  0.0016 &   11085   &  0.0029 \\
  \midrule
  \multicolumn{10}{c}{families of trees }\\
  \midrule
   $T_{8}$& 23   & 6  & 0.7391  & 6  & 0.7391    &  6 &  0.7391 & 19 & 0.7391   \\
  $T_{9}$& 47    & 39 &  0.1702 & 39  & 0.1702 & 39  &    0.1702   & 39 &  0.1702   \\
  $T_{10}$& 106& 83 &  0.2170 & 83 &0.2170 & 83 &  0.2170 & 83 &   0.2170 \\
  $T_{11}$&235 &221 & 0.0596 & 221  & 0.0596  & 221 & 0.0596  & 221 &  0.0596\\
  $T_{12}$& 551  & 528 &0.0417  &528  &0.0417  & 528 & 0.0417 &528  &0.0417  \\
  $T_{13}$& 1301 & 1286 & 0.0115 & 1286  & 0.0115  &1286  & 0.0115  &  1286& 0.0115  \\
  $T_{14}$& 3159&  3131 & 0.0089 & 3131 &0.0089  & 3131  & 0.0089 & 3131 &0.0089 \\
  $T_{15}$& 7741 & 7724& 0.0022 & 7724 & 0.0022 & 7724  & 0.0022  & 7724 &0.0022  \\
  $T_{16}$& 19320 & 19289& 0.0016  &19289  & 0.0016 & 19289 & 0.0016 & 19289 &  0.0016\\
  \bottomrule
\end{tabular}
\end{table}}

The results demonstrate that the considered root-indices outperform the standard versions of indices in terms of discrimination power, with this difference being particularly noticeable in families of trees. For instance, the Gutman root-index can discriminate more than 67\% of all connected graphs with 7 vertices, whereas the Gutman index can distinguish less than 17\% of these graphs. Remarkably, the mentioned root-index can distinguish almost 80\% of trees with 16 vertices, whereas the corresponding index demonstrates a significantly lower discrimination rate of less than 0.2\%.

We further consider  correlations between different pairs of root-indices. Two sets of samples are examined: all 11117 connected graphs on 8 vertices and the complete collection of 7741 trees on 15 vertices. The  Pearson correlation coefficients are collected in Table \ref{tab_cor-poly0}. It can be observed that the correlations are very strong, since the Pearson correlation coefficient in all cases is greater than 0.93.

\begin{table}[!ht]
\caption{\label{tab_cor-poly0} Correlation coefficients between root-indices in graph families $T_{15}$ and $N_8$.}
	\begin{tabular} {cccccc}\toprule
	graph family& root-index& $\delta(H^{\ast})$  &$\delta(Gut^{\ast})$  & $\delta(Sc^{\ast})$ & $\delta(H_e^{\ast})$\\ \midrule
\multirow{3}{*}{$T_{15}$}&	$\delta(H^{\ast})$   & \multirow{3}{*}{} & 0.9324  &   0.9872 &   0.9681  \\ \cmidrule{2-6}
&	$\delta(Gut^{\ast})$    &  \multicolumn{2}{c}{}  & 0.9479 &  0.9617\\   \cmidrule{2-6}
&	$\delta(Sc^{\ast})$   & \multicolumn{3}{c}{}   & 0.9945\\  \midrule
\multirow{3}{*}{$N_{8}$} &	$\delta(H^{\ast})$  & \multirow{3}{*}{} &  0.9440  &   0.9784  &   0.9591 \\  \cmidrule{2-6}
&	$\delta(Gut^{\ast})$   & \multicolumn{2}{c}{} & 0.9873 &  0.9953\\  \cmidrule{2-6}
&	$\delta(Sc^{\ast})$   & \multicolumn{3}{c}{}  & 0.9959\\  \bottomrule
	\end{tabular}
\end{table}

We also compared the corresponding standard topological indices and the correlation coefficients are gathered in Table \ref{tab_cor-poly1}. We notice that for the family $T_{15}$, there is almost linear relation between any pair of topological indices. On the other hand, there are slightly weaker correlations between them in the family $N_8$. It is also interesting that in $N_8$, the Wiener index negatively correlates with other standard topological indices.

\begin{table}[!ht]
\caption{\label{tab_cor-poly1} Correlation coefficients between standard  indices in graph families $T_{15}$ and $N_8$.}
	\begin{tabular} {cccccc}\toprule
	graph family& index& $W$  &$Gut$  & $Sc$ & $W_e$\\       \midrule
\multirow{3}{*}{$T_{15}$}&	$W$   & \multirow{3}{*}{} & 1.0000  &   1.0000 &   1.0000 \\ \cmidrule{2-6}
&	$Gut$    &  \multicolumn{2}{c}{}  & 1.0000 &  1.0000\\   \cmidrule{2-6}
&	$Sc$   & \multicolumn{3}{c}{}   & 1.0000\\ \midrule
\multirow{3}{*}{$N_{8}$} &	$W$  & \multirow{3}{*}{} &  -0.8598  &   -0.7646  &   -0.8518 \\   \cmidrule{2-6}
&	$Gut$   & \multicolumn{2}{c}{} & 0.9614 &  0.9933\\   \cmidrule{2-6}
&	$Sc$   & \multicolumn{3}{c}{}  & 0.9422\\  \bottomrule
	\end{tabular}
\end{table}

In addition, we computed correlations between original  indices and  corresponding root-indices in two families of graphs, see Table \ref{tab_cor-poly}. All the correlations are quite high, but the best one is obtained between the Wiener index and the Wiener root-index in the family of all connected graphs on 8  vertices. Moreover, it is interesting that in the same family of graphs, all other pairs of indices are  negatively correlated. 

\begin{table}[!ht]
\caption{\label{tab_cor-poly} Correlation coefficients between indices and root-indices in graph families $T_{14}$ and $N_8$.}
	\begin{tabular} {cccc}\toprule
graph class	&index & root-index & correlation coefficient \\ \midrule
\multirow{4}{*}{$T_{14}$} &$W$& 	$\delta(H^{\ast})$   &  0.8220\\  \cmidrule{2-4}
&	$Gut$& 	$\delta(Gut^{\ast})$    &  0.9418 \\  \cmidrule{2-4}
&$Sc$& 	$\delta(Sc^{\ast})$    & 0.8335  \\  \cmidrule{2-4}
 &$W_e$ & $\delta(H_e^{\ast})$  &  0.8579 \\  \midrule
 \multirow{4}{*}{$N_8$} &$W$& 	$\delta(H^{\ast})$   &  0.9579 \\  \cmidrule{2-4}
&	$Gut$& 	$\delta(Gut^{\ast})$    &   -0.7781\\  \cmidrule{2-4}
&$Sc$& 	$\delta(Sc^{\ast})$    &  -0.8248 \\  \cmidrule{2-4}
 &$W_e$ & $\delta(H_e^{\ast})$  &  -0.7996 \\  \bottomrule  
	\end{tabular} 
\end{table}

Finally, we investigate some quantities which measure how a gradual change of a graph results on the topological index. We need several additional concepts.

Let ${\cal F}$ be a family of connected graphs and let $G \in {\cal F}$. Furthermore, let \(\mathcal{S}(G)\) denote the collection of all pairwise non-isomorphic graphs that can be derived from \(G\) by the addition of precisely one edge. Consequently, the Graph Edit Distance (GED) between \(G\) and any graph within \(\mathcal{S}(G)\) is equal to one \cite{gao,sanfeliu}.

The \textit{structure sensitivity}  of a topological index $I$  for the graph $G$, $SS^1_G(I)$, is defined in the following way \cite{furtula,rakic}:
$$SS^1_G(I)=\displaystyle\frac{1}{|{\cal S}(G)|}
\sum_{G' \in {\cal S}(G)}\left|  \frac{I(G)-I(G')}{I(G)}\right|.$$

\noindent
Moreover, the \textit{abruptness} of  $I$  for the graph $G$, $Abr^1_G(I)$, is calculated as
$$Abr^1_G(I)=\max_{G' \in {\cal S}(G)}
\left|  \frac{I(G)-I(G')}{I(G)}\right|.$$


Furthermore, in \cite{rakic} the authors proposed different versions of  structure sensitivity, $SS_G^2(I)$, and  abruptness, $Abr_G^2(I)$, of a graph $G$ for a topological index $I$:
\begin{eqnarray*}
SS^2_G(I) & = & \displaystyle\sqrt{\displaystyle\frac{1}{|{\cal S}(G)|}
\sum_{G' \in {\cal S}(G)}  (I(G)-I(G'))^2},\\
Abr^2_G(I) & = & \max_{G' \in {\cal S}(G)}
\left| I(G)-I(G')\right|.
\end{eqnarray*}

The structure sensitivity and abruptness of a topological index $I$ on a family of graphs ${\cal F}$ is defined as the average over all elements of ${\cal F}$. So, for $i \in \{ 1,2 \}$, we have
\begin{eqnarray*}
SS^i(I) & = & \frac{1}{|{\cal F}|} \sum_{G \in {\cal F}} SS^i_G(I),\\
Abr^i(I) & = & \frac{1}{|{\cal F}|} \sum_{G \in {\cal F}} Abr^i_G(I).
\end{eqnarray*}

It is preferred  for a topological index to have the structure sensitivity as large as possible and the abruptness as small
as possible \cite{furtula}. If a topological index fulfils this property, on the one hand it can distinguish well between similar graphs, but on the other hand there is not much difference between the changes in the topological index. Therefore, it seems reasonable to introduce a new measure $SA^i$, $i \in \{1,2 \}$, as
$$SA^i(I) = \frac{SS^i (I)}{Abr^i(I)}.$$
The new measure enables us to compare different topological indices, since it is desired for the quotient $SA^i$, $i \in \{1,2 \}$, to be as large as possible. The results for both versions of structure sensitivity, abruptness, and their quotient of root-indices for the class of trees with $n$ vertices, $n \in \{ 9,10,11,12 \}$, are shown in Table \ref{tab_ss}.

\begin{table}[!ht]
\caption{\label{tab_ss} Structure sensitivity and abruptness of root-indices for all trees on $n$ vertices, where $n \in \{ 9,10,11,12 \}$.}
	\begin{tabular} {cccccccc}\toprule
	root-index &graph class & $SS^1$ & $Abr^1$ & $SA^1$ & $SS^2$ & $Abr^2$ & $SA^2$ \\ \midrule
\multirow{4}{*}{$\delta(H^{\ast})$} & $T_9$ &
 0.0977 & 0.1161   & 0.8417    &
 	0.0106   & 0.0125  & 0.8471   \\  \cmidrule{2-8}
	& $T_{10}$ &
0.0902	& 0.1080 & 0.8352 &
		0.0088    & 0.0104  & 0.8405   \\  \cmidrule{2-8}
& $T_{11}$&
 0.0837& 0.1010 & 0.8288 &
	0.0074    & 0.0089 & 0.8339     \\  \cmidrule{2-8}
&  $T_{12}$ &
0.0779 & 0.0943 & 0.8264 &
 0.0063  & 0.0076  & 0.8312   \\  \midrule

\multirow{4}{*}{$\delta(Gut^{\ast})$} & $T_9$&
 0.2794& 0.3853 & 0.7252 & 	
 0.0074   & 0.0099  & 0.7416   \\  \cmidrule{2-8}
	& $T_{10}$ &
0.2567	& 0.3708 & 0.6922 &
		0.0058    & 0.0082  & 0.7123   \\  \cmidrule{2-8}
& $T_{11}$&
 0.2381& 0.3591 & 0.6632 &
 	0.0048    & 0.0069 & 0.6857    \\  \cmidrule{2-8}
&  $T_{12}$ &
 0.2218 & 0.3470 & 0.6391 & 0.0040  & 0.0060  & 0.6637   \\  \midrule

\multirow{4}{*}{$\delta(Sc^{\ast})$} & $T_9$&
 0.1845 & 0.2622 & 0.7038 &
 	0.0049   & 0.0068  & 0.7287   \\  \cmidrule{2-8}
	& $T_{10}$
& 0.1696	& 0.2480 & 0.6839
	& 	0.0040    & 0.0056  & 0.7103   \\  \cmidrule{2-8}
& $T_{11}$&
 0.1569& 0.2362 & 0.6646 &
 	0.0033    & 0.0048 & 0.6917    \\  \cmidrule{2-8}
&  $T_{12}$ &
0.1459& 0.2244 & 0.6501 &
 0.0028  & 0.0040  & 0.6776   \\  \midrule

\multirow{4}{*}{$\delta(H_e^{\ast})$} & $T_9$&
 0.2288 & 0.3373 & 0.6784 &
 	0.0212   & 0.0298  & 0.7120   \\  \cmidrule{2-8}
	& $T_{10}$ &
 0.2111	& 0.3218 & 0.6560&
		0.0170    & 0.0246  & 0.6915   \\  \cmidrule{2-8}
& $T_{11}$&
 0.1960 & 0.3088 & 0.6349 &
 	0.0140    & 0.0208 & 0.6711    \\  \cmidrule{2-8}
&  $T_{12}$ &
0.1828 & 0.2954 & 0.6187 &
 0.0117  & 0.0178 & 0.6554  \\  \bottomrule
	\end{tabular}
\end{table}

The results show that the best index regarding $SA^1$ and $SA^2$ is the Wiener root-index. For example, for trees on 11 vertices we get $SA^1(\delta(H^*)) = 0.8288$, $SA^1(\delta(Gut^*)) = 0.6632$, $SA^1(\delta(Sc^*)) = 0.6646$, $SA^1(\delta(H_e^*)) = 0.6349$ and similarly $SA^2(\delta(H^*)) = 0.8339$, $SA^2(\delta(Gut^*)) = 0.6857$, $SA^2(\delta(Sc^*)) = 0.6917$, $SA^2(\delta(H_e^*)) = 0.6711$. We also observe that on trees both measures, $SA^1$ and $SA^2$, behave very similarly for all four indices. Therefore, it seems that  if we want to evaluate the performance of different topological indices on trees, it does not matter with which version we work.

\section{Summary and conclusion}

 We introduced and examined several novel root-indices of graphs, which can be used as quantitative graph measures. Initially, we presented analytical findings related to the roots of  new modified graph polynomials. More precisely, several closed formulas and bounds were determined.

 In the final section, we performed quantitative analysis and compared the results to original topological descriptors. The results indicate that the new descriptors have significantly enhanced discrimination power. Among connected graphs, the Gutman root-index shows the highest discrimination power. However, for tree families, all four root-indices exhibit almost identical, high discrimination power. Moreover, it can  be observed that the correlations between different root-indices are very strong, in all cases greater than 0.93. Additionally, the correlations between root-indices and their original versions are also quite high, which means that the root-indices capture information similarly as other standard (existing) graph measures.
On the other hand, it seems better to use root-indices since they have much higher discrimination power. 

Regarding structure sensitivity and abruptness, which quantify how a gradual change of a graph results on the topological index, we introduced their quotient as a new measure in order to simplify the comparison of different topological indices. We found out that the best index regarding the mentioned new measure is the Wiener root-index.

{As a future work, we would like to improve the stated bounds by finding novel estimations.
We are convinced that these bounds can be generalized such that the presented bounds will result as special cases.
Also, we want to prove more upper and lower bounds for these measures by considering similar or other graph polynomials and compare
the results with the ones we have stated in this paper. Another idea is to prove so-called implicit bounds \cite{dehmer_mowshowitz_bounds_2011}. } Additionally, we would like to investigate chemical applications of various root-indices.




\section*{Funding information} 

\noindent Simon Brezovnik, Niko Tratnik, and Petra \v Zigert Pleter\v sek acknowledge the financial support from the Slovenian Research and Innovation Agency: research programme No.\ P1-0297 (Simon Brezovnik, Niko Tratnik, Petra \v Zigert Pleter\v sek), projects No.\ J1-4031 (Simon Brezovnik), N1-0285 (Niko Tratnik), and L7-4494 (Petra \v Zigert Pleter\v sek).

\section*{Conflict of Interest Statement}

Not Applicable. The author declares that there is no conflict of interest.



\end{document}